\documentclass[12pt,reqno]{amsart}
\usepackage{amsmath,amssymb}
\usepackage[dvips]{graphicx,color}
\usepackage{latexsym,amssymb}
\usepackage{mathrsfs}
\usepackage[abbrev]{amsrefs}

\newtheorem{thm}{Theorem}[section]

\newtheorem{prop}[thm]{Proposition}
\newtheorem{lem}[thm]{Lemma}

 \newenvironment{pf}
    {{\noindent \bf Proof. }}{\hfill $\Box$}

\numberwithin{equation}{section}
\numberwithin{thm}{section}

\begin{document}

\begin{center}\large \bf 
Besov space approach to the Navier-Stokes equations 
with the Neumann boundary condition in bounded domains
\end{center}

\footnote[0]
{
{\it Mathematics Subject Classification}: 76D05; 35Q30

{\it 
Keywords}: 
Navier-Stokes equations, 
bounded domain, 
Neumann condition

E-mail: t-iwabuchi@tohoku.ac.jp~(T. Iwabuchi), 
kozono@waseda.jp~(H. Kozono). 

}

\begin{center}

Tsukasa Iwabuchi$^{\rm a}$, and Hideo Kozono$^{\rm b,c}$

\vskip2mm

$^{\rm a}$ 
Mathematical Institute, 
Tohoku University 
Sendai 980-8578 Japan

$^{\rm b}$ 
Department of Mathematics, Waseda University, 169-8555 Tokyo, Japan

$^{\rm c}$ 
Mathematical Research Center for Co-creative Society, Tohoku University, 980-8578 Sendai, Japan

\end{center}

\vskip5mm

\begin{center}
\begin{minipage}{135mm}
\footnotesize
{\sc Abstract. } 
Based on the analysis by Iwabuchi-Matsuyama-Taniguchi \cite{IMT-2019}, 
\cite{IMT-2018}, we first introduce our framework of Besov spaces  
$\dot B^s_{p, q}$ for $s \in {\mathbb R}$ and $1 \le p, q \le \infty$ 
on the bounded domain $\Omega \subset {\mathbb R}^d$ with 
smooth boundary $\partial \Omega$ in terms 
of the Stokes operator $A=A_2$ with the Neumann boundary condition 
on $\partial\Omega$ 
in  $L^2_{\sigma}(\Omega)$. 
Under some geometric assumption on $\Omega$, we establish $L^p-L^q$ type estimates of the semi-group 
$\{e^{-tA}\}_{t \ge 0}$ in $\dot B^s_{p, q}$ and prove a local well-posedness 
of the Navier-Stokes equations with the initial data in $\dot B^{-1+\frac dp }_{p, q}$ 
for $d < p < \infty$ and $1 \le q \le \infty$.  Since $d < p$, we have 
$L^{d, \infty} \subset \dot B^{-1+\frac dp }_{p, \infty}$ 
so that our space for well-posedness is larger than any other previous one 
in bounded domains.        
\end{minipage}
\end{center}

\section*{Introduction}

Let $\Omega$ be a bounded domain in 
$\mathbb R^d$ ($d \geq 2$) with smooth boundary $\partial\Omega$.  
We assume that the $d-1$-dimensional Betti number of $\Omega$ is zero.  
We consider the incompressible Navier-Stokes equations 
in $\Omega$ with the Neumann boundary condition on $\partial\Omega$; 
\begin{equation}\label{NS}
\begin{cases}
\displaystyle 
 \partial_t u - \Delta u
  + (u \cdot \nabla ) u  + \nabla \pi = 0, 
  & t > 0 , x \in \Omega, 
  \\
{\rm div }\, u = 0 , 
  & t> 0 , x \in \Omega , 
  \\
  \sigma(\delta , \nu)u =0, \quad
  \sigma(\delta, \nu)du=0 ,
  & t > 0 , x \in \partial\Omega , 
\\
u(0,x) = u_0,  
  & x \in \Omega, 
\end{cases}
\end{equation}
where $\nu$ is the unit outer normal vector. 
In (\ref{NS}), $\sigma(\delta, \nu)$ denotes the value at $\nu$ of the 
principal symbol $\delta = d^\ast$ with $d$ denoting the exterior differentiation 
acting on differential forms on $\Omega$.  
In the case when $d=3$, this means that 
\[
u \cdot \nu = 0 , \quad {\rm rot \,} u \times \nu = 0 \quad \text{ on } \partial \Omega . 
\]
Notice that the unknown vector field $u=u(x, t) = (u_1(x, t), \cdots, u_d(x, t))$ 
may be regarded as the 1-form $u = \sum_{j=1}^d u_j(x, t)dx_j$ 
on $x=(x_1, \cdots, x_d) \in \Omega$ and $t >0$, and that 
$\pi= \pi(x, t)$ is the unknown scalar function denoting the pressure, 
while $u_0=u_0(x) = (u_{0, 1}(x), \cdots, u_{0, d}(x))$ is the given initial 
velocity vector field.  
\par
The purpose of this paper is to show well-posedness of (\ref{NS}) in a larger 
space than $L^d(\Omega)$.  In the case when $\Omega = \mathbb R^d$ there 
are a number of literature on well-posedness.  
As a pioneer work, Leray \cite{Le-1934} constructed a global weak solutions 
$u \in L^\infty(0, T;L^2_\sigma(\mathbb R^3)) \cap L^2(0, T; H^1(\mathbb R^3))$ 
of (\ref{NS}) for every $u_0 \in L^2_{\sigma}(\mathbb R^3)$, 
where $T>0$ is taken arbitrarily large.   
Later on,  by Serrin \cite{Se-1962,Se-1963}, Prodi \cite{Pr-1959}, Masuda \cite{Mas-1984},  
Kozono-Sohr \cite{KoSo-1997,KoSo-1996} and Escauriaza-Seregin-Sver\'ak \cite{EsSeSh-2003},   
uniqueness and regularity of such a weak solution $u$ are shown 
provided $u \in L^s(0, T; L^p(\mathbb R^d))$ for $2/s + d/p = 1$ with 
$d\le p \le \infty$.  
Such a class $L^s(0, T; L^p(\mathbb R^d))$ is scaling invariant for (\ref{NS}), 
and in particular the marginal case $s=\infty$ and $p=d$ has been fully investigated by many authors.   Indeed, Kato \cite{Ka-1984}, Cannone-Planchon \cite{CaPl-1996}, Kozono-Yamazaki \cite{KoYam-1994} and Koch-Tartar \cite{KoTa-2001}  
proved local well-posedness in $L^d(\mathbb R^d)$, $\dot B^{-1+ \frac dp}_{p, q}(\mathbb R^d)$, $M^{-1+ \frac dp}_{p, q}(\mathbb R^d)$ and $BMO^{-1} = \dot F^{-1}_{\infty, 2}(\mathbb R^n)$, where the symbol  $M$ and $F$ denote 
the Morrey and the Triebel-Lizorkin spaces, respectively 
(see also \cite{IwNa-2013}).   
The case when $p=\infty$, Bourgain-Pavlovic \cite{BouPa-2008}, Yoneda \cite{Yo-2010} 
and Wang \cite{Wa-2015} proved ill-poseness in $\dot B^{-1}_{\infty, q}(\mathbb R^d)$ 
for all $1 \le q \le \infty$.  
For the proof of well and ill-posedness, the techniques such as 
paraproduct formula, Littlewood-Paley decomposition, norm-inflation, etc. 
arising from Besov spaces play an essential role in the special case 
when $\Omega = \mathbb R^d$.  
\par
In the case when $\Omega$ is a bounded domian, 
there are less results on well-posedness of (\ref{NS}) 
in comparison with that in $\mathbb R^d$.  
Fujita-Kato \cite{FuKa-1964} treated homogeneous Dirichelt condition, 
i.e., $u|_{\partial \Omega}=0$ and proved local well-poseness in 
the 3D case in the domain  
$D(A_2^{1/4})$ of the fractional power $A_2^{1/4}$ of the Stokes operator 
$A_2$ in $L^2_\sigma(\Omega)$.  
Giga-Miyakawa \cite{GiMi-1985} had been successful to generalize the result of 
\cite{FuKa-1964} in the $d$-dimensional case and proved local well-posedness in 
$L^d_\sigma(\Omega)$.  Up to the present, their result on $L^d(\Omega)$ is 
the largest space which guarantees local well-posedness.  
\par
In the paper we consider the Navier-Stokes equations in bounded domain $\Omega 
\subset \mathbb R^d$ with the Neumann condition on  $\partial \Omega$. 
Under a certain geometric assumption on $\Omega$, we first define the 
Besov space $B^s_{p, q}(\Omega)$ by means of the Stokes operator $A_2$ 
on $L^2_\sigma(\Omega)$ with the Neumann boundary condition. 
Our method is based on the procedure of the Dirichlet 
Laplacian given by the series of papers 
\cite {Iw-2018}, \cite{Iw-2023}, \cite{IMT-2021}, 
\cite{IMT-2019} and \cite {IMT-2018}.
We first construct a partition $\{\phi_j(\sqrt{A_2})\}_{j\in \mathbb Z}$ 
of unity to define the Littlewood-Paley decomposition in $L^2(\Omega)$ associated with $A_2$.  
Since $A_2$ is a positive self-adjoint operator in $L^2_\sigma(\Omega)$, 
we need to extend $\{\phi_j(\sqrt{A_2})\}_{j \in \mathbb Z}$ to a family of  bounded operators
in $L^p_\sigma(\Omega)$ for all $1 \leq p \leq \infty$.  
To this end, 
the pointwise estimate of the kernel function of the semi-group 
$\{e^{-tA_2}\}_{t>0}$ in terms of the Gaussian plays an essential role. 
Since the Laplacian acting on 1-forms with the Neumann condition on 
$\partial\Omega$ commutes with 
the Helmholtz-Weyl projection $\mathbb P$, we are successful to establish such a 
pointwise estimate so that the uniform $L^p$-bound of 
$\{\phi_j(\sqrt{A_2})\}_{j \in \mathbb Z}$ is obtained 
in the same way as the case of the Dirichlet Laplacian. 
We next show the $L^p$-$L^q$-type estimates of $\{e^{-tA}\}_{t>0}$ in 
$\dot B^s_{p, q}$ and the bilinear estimate in the scaling space 
$\dot B^{-1+ \frac np}_{p, \infty}$ with $d < p < \infty$ so that 
(\ref{NS}) is locally  well-posed in such a space.   
Since $d < p < \infty$ it holds $L^{d, \infty}(\Omega) \subset \dot B^{-1+ \frac np}_{p, \infty}$, which means that our result on well-posedness 
is lager than that of Miyakawa \cite {Miy-1980} and Giga-Miyakawa \cite{GiMi-1985}.  
\par
\bigskip
This paper is organized as follows. 
In the next section we give a precise definition of the space 
$\dot B^s_{p, q}$ and then state our main theorem,  
Section 2 is devoted to several properties of the Stokes operator $A$
with the Neumann condition on $\partial \Omega$ in $B^s_{p, q}(\Omega)$.  
In particular, we establish  $L^p-L^q$-type estimates of $\{e^{-tA}\}_{t>0}$ in 
such Besov spaces. Finally in Section~3  we prove our main theorem.

\par
\bigskip 
\section{Result}
We impose the following assumption on $\Omega$. 
\par
{\bf Assumption.} 
$$
X_{\mbox{\tiny har}}(\Omega)\equiv \{u \in C^\infty(\bar \Omega); 
\delta u =0, \, \, du= 0, \,\,\, 
\sigma(\delta, \nu)u =0 \quad\mbox{on $\partial\Omega$} \} 
= \{0\}.    
$$  
\par
In the case $d=2, 3$  the Assumption states $\Omega$ is simply 
connected (see \cite{KoYa-2009}).

We will define the Stokes operator $A$ on a distribution spaces denoted by 
$\mathcal Z'_\sigma$ 
and introduce the Besov spaces associated with the Stokes operator 
in subsection~\ref{subsec:Besov}. 
Let $\phi $ satisfy that 
\[  \phi \in C_0^\infty (\mathbb R), \quad 
{\rm supp \,} \phi \in [2^{-1},2] \quad
0 \leq \phi \leq 1 , \quad 
\sum_{j \in \mathbb Z} \phi (2^{-j}\lambda ) = 1, \quad \lambda >0. 
\]
We introduce $\{ \phi_j \}_{j \in \mathbb Z} $ such that for every $j \in \mathbb Z$
\[
\phi_j(\lambda) := \phi(2^{-j} \lambda), \quad \lambda \in \mathbb R, 
\]
which is a partition of the unity by dyadic numbers 
on the half line $(0,\infty)$. 
Formally, the norm  is written by 
\[
\| u \|_{\dot B^s_{p,q}} 
:= \Big\| \Big\{ 2^{js} \| \phi_j(\sqrt{A}) u \|_{L^p} \Big\}_{j \in \mathbb Z} \Big\|_{\ell^q (\mathbb Z)} ,
\]
where $\phi_j (\sqrt{A}) $ is defined in subsection~\ref{subsec:StokesOp}.

\vskip3mm 

\noindent {\bf Definition. }
Let $X$ be a Banach space, $T>0$ and $u_0 \in X$. 
We say $u$ a solution to \eqref{NS} in $X$ on $(0,T)$ if 
$u \in L^\infty (0,T ; X)$ 
and if $u$ satisfies the following integral equation; 
\[
u(t) = e^{tA}u_0 - \int _0^t e^{-(t-\tau)A} 
  \mathbb P \nabla \cdot (u \otimes u)  ~{\rm d}\tau 
  \quad \text{ in } X , 
\]
for almost every $t \in (0,T)$.

\begin{thm}\label{thm}
Let $d \geq 2$ and $d < p < \infty$. 
\begin{enumerate}
\item {\rm(}Local solutions{\rm)} 
Let $1 \leq q < \infty$ and $u_0 \in \dot B^{-1+\frac{d}{p}}_{p,q}$. 
Then there exist $T > 0$ and a solution $u$ to \eqref{NS} 
in $\dot B^{-1+\frac{d}{p}}_{p,q}$ on $(0,T)$ such that 
\[
u \in C([0,T], \dot B^{-1+\frac{d}{p}}_{p,q}), \quad 
\sup _{ t \in (0,T)}t^{\frac{1}{2}(1-\frac{d}{p})}  
  \| u(t) \|_{L^p} < \infty .   
\]

\item {\rm(}Local solutions{\rm)}  Let $q = \infty$. 
There exists  a positive constant $\delta _0 $ such that 
if $u_0 \in \dot B^{-1+\frac{d}{p}}_{p,\infty} $ satisfies 
\[
\limsup_{j \to \infty} 2^{(-1+\frac{d}{p})j} 
 \| \phi_j (\sqrt{A}) u_0 \|_{L^p} \leq \delta _0, 
\]
then there exist $T>0$ and  a unique solution $u$ to 
\eqref{NS} in $\dot B^{-1+\frac{d}{p}}_{p,\infty}$ on $(0,T)$ satisfying 
\[
\sup_{t \in (0,T)} t^{\frac{1}{2}(1-\frac{d}{p})} 
\| u(t) \|_{L^p} \leq C \delta _0 ,
\]
where $C = C(d,p)$ is independent of $u_0$. 
The continuity in time holds in the dual weak sense that 
 for every $\tilde t \in [0,T]$, 
  $u(\tilde t)$ is an element of $\dot B^{-1+\frac{d}{p}}_{p,\infty}$ 
 and 
\[
\lim_{t \to \tilde t}
\sum_{j \in \mathbb Z}\int_{\Omega} 
\Big( \phi_j(\sqrt{A})\big( u(t) - u(\tilde t)\big)  
\Big) 
\overline{\Phi_j(\sqrt{A})g} ~\mathrm{d}x
= 0 , \quad 
\text{ for any } g \in \dot B^{1-\frac{d}{p}}_{p',1},
\]
where $1/ p + 1/p' = 1$. 

\item {\rm(}Global solutions{\rm)} 
In the case when $1 \leq q < \infty$, 
there is a positive constant $\delta _ 0 $ such that if 
$u_0 \in \dot B^{-1+\frac{d}{p}}_{p,q} \leq \delta _ 0$, 
then there exists a unique solution $u$ to \eqref{NS} in 
$\dot B^{-1+\frac{d}{p}}_{p,\infty}$ on $(0,\infty)$ such that 
\[
u \in C([0,\infty), \dot B^{-1+\frac{d}{p}}_{p,q}) 
\quad \text{ with } 
\sup _{t>0} t^{\frac{1}{2}(1-\frac{d}{p})} \| u(t) \|_{L^p} < \infty .
\]
In the case when $q = \infty$, there is a positive constant $\delta _ 0$ 
such that if 
$u_0 \in \dot B^{-1+\frac{d}{p}}_{p,\infty}$ and 
$\| u_0 \|_{\dot B^{-1+\frac{d}{p}}_{p,\infty}} \leq \delta _0$, 
then there exists a unique solution $u$ to
 \eqref{NS} in 
$\dot B^{-1+\frac{d}{p}}_{p,\infty}$ on $(0,\infty)$ satisfying 
\[
\sup_{t > 0} t^{\frac{1}{2}(1-\frac{d}{p})} \| u(t) \|_{L^p} \leq C\delta_0,
\]
where $C = C(d,p)$ is independent of $u_0$. 
Such continuity in the dual weak sense as in (2) holds for this solution. 
\end{enumerate}

\end{thm}

\noindent {\bf Remarks. } 
(i) If $d < p$, then we have that 
$L^d \subset L^{d,\infty} \subset \dot B^{-1+\frac{d}{p}}_{p,\infty}$. 
Hence so far as bounded domain is concerned, 
our space $\dot B^{-1+\frac{d}{p}}_{p,\infty}$ with $d < p$ for the initial 
data is larger than any other one given by Fujita-Kato \cite {FuKa-1964}, 
Miyakawa \cite{Miy-1980}  and Giga-Miyakawa \cite{GiMi-1985}.  \par
(ii) In case $X_{\mbox{\tiny har}}(\Omega) \ne \{0\}$, zero is an 
eigenvalue of $A$ so that it is difficult to obtain a global solution 
even for small initial data such as Theorem \ref {thm} (3).  
\par
(iii) It seems an interesting problem whether a similar local and well-posedness 
to our result does hold in the case of the Dirichlet boundary condition 
$u =0$ on $\partial\Omega$.  Since the Neumann boundary condition yields 
commutativity between the Laplace operator  $-\Delta$ and the Helmholtz projection    
$\mathbb P$, we may show the pointwise estimate of the kernel function of the 
semigroup $e^{-tA}$, which makes it possible to develop the $L^p$-theory 
starting from the $L^2$-setting. 
\par  
(iv) Concerning the Euler equations for the ideal incompressible fluid in 
${\mathbb R}^d$,  the local well-posedness has been proved for 
$u_0 \in \dot B^{\frac dp +1}_{p, 1}$, $1 \le p \le \infty$ by 
 Chae \cite{Ch-2002} and Pak-Park \cite{Pa-2004}.   
On the other hand, in bounded domain, Kato-Lai \cite{KaLa-1984} had obtained 
the corresponding result in the Sobolev space $u_0 \in H^{s, 2}$ 
with $s = [d/2] + 2$.  It seems an interesting question whether 
our method does work to show the local well-posedness in a larger space than 
that of \cite{KaLa-1984}. We will discuss it in a forthcoming paper.

\section{Preliminary}  

We introduce several key definitions that are essential to addressing our problem. 
These include the definition of distributions derived from the spectral 
decomposition of the Stokes operator on $L^2$, the projection of these 
distributions onto divergence-free vector fields, and the Besov spaces associated 
with the Stokes semigroup. 
These concepts form the foundation of our argument.

In subsection~\ref{subsec:basic}, we review the Stokes operator on $L^2$, 
the projection $\mathbb P_2$ from $L^2$ to $L^2_\sigma$ (the spaces of 
divergence-free vector fields), 
and related elliptic estimates. 
In subsection~\ref{subsec:StokesOp}, we introduce several spaces 
for our test functions and distributions following the approach in 
\cite{FI-2024,IMT-2019}. 
This allows us to define the Stokes operator in these distribution spaces. 
In subsection~\ref{subsec:projection}, we introduce 
the projection $\mathbb P$ on the distribution spaces. 
In subsection~\ref{proj_nonlinear}, 
we define the projection for the nonlinear term 
${\rm div \, } (u \otimes u) $. 
In subsection~\ref{subsec:Besov}, 
we define Besov spaces associated with the Stokes operator in 
the distribution spaces 
introduced in subsection~\ref{subsec:StokesOp}. 
Finally, in subsection~\ref{subsec:StokesSemi}, 
we present several estimates for the Stokes semigroup,  
together with some inequalities involving the gradient, which are 
analogous to the elliptic estimates. 

\subsection{Basic Facts}\label{subsec:basic}

We recall several facts for divergence-free vector fields, 
and the Stokes operator on $L^p$. 
We introdue the set of smooth functions and the Lebesgue space. 
\[
\begin{split}
& C^\infty_{0,\sigma} := \{ u \in C^\infty (\Omega) \, | \, 
{\rm div \,} u = 0, \,\, { \rm supp \, } u \subset \Omega : \text{compact} \}, 
\\
& L^p_\sigma := \overline{ C_{0,\sigma}^\infty (\Omega) }^{\| \cdot \|_{L^p}}. 
\end{split}
\]
We define the Stokes operator $A_p$ on $L^p_\sigma  (\Omega)$ by 
\[
\begin{cases}
D(A_p) = \{ u \in W^{2,p}(\Omega) \, | \, 
\sigma (\delta ,\nu) u = 0, 
\sigma(\delta , \nu) du= 0 \text{ on } \partial \Omega \}
   \cap L^p_\sigma. 
\\
A_p u = - \mathbb P_p\Delta u = - \Delta u , 
\quad u \in D(A_p). 
\end{cases}
\]

In the $L^p$ setting, the solvability of the following Neumann problem is well-known. 
\[
\begin{cases}
\Delta \pi  = {\rm div \, } u
& \text{ in } \Omega , 
\\
\dfrac{\partial \pi}{\partial \nu} = u \cdot \nu 
& \text{ on } \partial \Omega . 
\end{cases}
\]
In fact, let us consider the weak Neumann problem. It is known that  
for the given $u \in L^p$, 
there exists a unique $\pi \in H^{1}_p$ such that 
\[
\langle \nabla \pi , \nabla \phi \rangle 
= \langle u, \nabla \phi \rangle , \quad 
\text{ for all } \phi \in H^1_{p'}, 
\]
where $\langle \cdot , \cdot \rangle$ denotes the dual coupling between $L^p$ and $L^{p'}$. 
The above property is equivalent to the existence of 
a positive constant $C$ exists such that 
\[
\| \nabla \pi \|_{L^p} 
\leq C \sup_{\phi \in H^1_{p'}} 
  \dfrac{|\langle \nabla \pi , \nabla \phi \rangle|}{\| \nabla \phi \|_{L^{p'}}}. 
\]
We then define the projection $\mathbb P_p : L^p \to L^p_\sigma$ by 
\[
\mathbb P_p u = u - \nabla \pi ,
\]
where $\nabla \pi$ is the solution to the weak Neumann problem described above. 


\begin{prop}\label{prop:0809-1}
\begin{enumerate}
\item 
{\rm(}Elliptic estimates{\rm )} 
There exists a positive constant $C$ such that for every $f \in L^p_\sigma$ 
\[
\| \nabla^2 A_p^{-1} f \|_{L^p} \leq C \| f \|_{L^p}.
\]
\item Let $1 < p < \infty$. 
 There exists a positive constant $C $ such that 
for every $f \in L^p$,  
\[
\| \nabla A_p^{-\frac{1}{2}} f \|_{L^p} \leq C \| f \|_{L^p}.  
\]
\end{enumerate}
\end{prop}

\noindent 
{\bf Remark. } We refer to the introduction in the paper~\cite{IMT-2021}, 
which mentions existing results about derivative estimates on domains. 

\vskip3mm 

\begin{pf}
(1) 
We refer to the papers by Agmon, Diuglis, and Nirenberg~\cite{ADN-1964} and Kozono and Yanagisawa~\cite{KoYa-2009} (see~Lemma~4.4 (2)) for the proofs. 

\noindent 
(2) This is proved using complex interpolation, as in the argument of the proof of Lemma~4.2 in \cite{Miy-1980}.
\end{pf}

\vskip3mm 

The operator $A_2$ is a self-adjoint operator on $L^2_\sigma$. 
By the spectral theorem, a resolution $\{ E(\lambda) \}_{\lambda \in \mathbb R}$ 
of the identity exists such that 
\[
f = \int_{-\infty}^\infty dE(\lambda ) f \quad \text{ in } L^2_\sigma , 
\quad f \in L^2_\sigma, 
\]
and 
$\varphi (A_2)$ is well-defined as an operator on $L^2_\sigma$ by 
\[
\varphi (A_2) f = \int_{-\infty}^\infty \varphi(\lambda) dE(\lambda) f , 
\]
where $\varphi : \mathbb R \to \mathbb C$ is an arbitrary measurable function.

\subsection{Distribution Spaces and the Stokes Operator}\label{subsec:StokesOp}
We extend $\varphi (A_2)$ on $L^2$ to operators on distribution spaces. 
To this end, we need the Gaussian upper bounds for the semigroup 
(see the appendix in~\cite{Miy-1980} and the paper \cite{RaSi-1971}). 

\begin{lem}\label{lem:0726-1}
Let $u_0 \in L^2_\sigma$ and $u(t) = e^{-t A_2} u_0 $. Then 
the kernel $e^{-tA_2}(x,y)$ of $ e^{-t A_2}$ satisfies 
that there exists a positive constant $C$ such that 
\[
|e^{-tA_2}(x,y)| \leq \dfrac{C}{t^{\frac{d}{2}}}e^{-\frac{|x-y|^2}{Ct}}, 
\]
for all $t >0$ and almost every $x,y \in \Omega$. 
\end{lem}

The lemma above and the ellipitic estimate on $L^2$ in 
Proposition~\ref{prop:0809-1} lead us to 
the extension of the 
spectral restriction operator $\phi_j(\sqrt{A_2})$ 
to a bounded operator on $L^p$ for all $1 \leq p \leq \infty$ 
together with the uniformity with respect to $j \in \mathbb Z$, 
 using the argument of \cite{IMT-2018,JeNa-1995} 
(see also \cite{Ou_2005,ThOuSi-2002}). 

We briefly review the argument presented in \cite{IMT-2018, JeNa-1995}. 
The $L^1$ estimate plays a crucial role, as duality provides 
the $L^\infty$ estimate, and complex interpolation extends the result 
to all $L^p$ spaces for $1 \leq p \leq \infty$. To establish the $L^1$ estimate, 
we decompose the domain $\Omega$ into cubes with side length $t^{-\frac{1}{2}}$ 
and apply the Cauchy-Schwarz inequality on each cube. This approach leads 
to the challenge of managing the localization via the $L^2$ norm and 
the non-local operator $\phi_j(\sqrt{A_2})$. The Gaussian upper bound 
in Lemma~\ref{lem:0726-1} ensures integrability 
(specifically, the convergence of the sum over all cubes), 
while expressing 
$\phi_j(\sqrt{A_2})$ using a resolvent function and partial derivatives 
requires the use of elliptic estimates. Ultimately, we obtain 
the boundedness of the operator norm $\|\phi_j(\sqrt{A_2})\|_{L^1 \to L^1}$ 
in weighted Sobolev spaces (see also Section 8 of \cite{IMT-2018}).
We here notice that the boundedness holds, including the end-point cases 
$p = 1$ and $p =\infty$. 

\vskip3mm 

The following proposition serves as the foundation for extending
$L^2$ theory to $L^p$ spaces and Besov spaces. 

\begin{prop}\label{prop:0726-2} {\rm (}\cite{IMT-2018}{\rm )}
Let $1 \leq p \leq \infty$, $\phi \in C_0^\infty (\mathbb R)$, 
and $\phi_j(\lambda ) = \phi (2^{-j}\lambda) $ for $j \in \mathbb Z $ 
 and $\lambda \in \mathbb R$. 
Then we have 
\[
\displaystyle 
\sup _{j \in \mathbb Z} \| \phi_j(\sqrt{A_2}) \|_{L^p \to L^p} < \infty. 
\]
Furthermore, if $1 \leq r \leq p \leq \infty$, then 
\[
\displaystyle 
\sup _{j \in \mathbb Z} 2^{-d(\frac{1}{r}-\frac{1}{p})j}\| \phi_j(\sqrt{A_2}) \|_{L^r \to L^p} < \infty . 
\]
\end{prop}

\begin{pf}
The proof is analogous to the proofs in \cite{IMT-2018,JeNa-1995}. 
\end{pf}

\vskip3mm 

Following the procedure in the paper~\cite{IMT-2019} (see also \cite{FI-2024}), 
we introduce test function spaces $\mathcal X_\sigma$ and $\mathcal Z_\sigma$, 
which are Fr\'echet spaces,  
and spaces of distributions as the topological duals of these. 
In the Euclidean space, we can regard that 
$\mathcal X_\sigma$ corresponds to the Schwartz class $\mathcal S(\mathbb R^d)$, 
defined by 
\[
\mathcal S(\mathbb R^d) 
= \{ f \in C^\infty (\mathbb R^d) \, | \, 
  \sup_{|\alpha|, |\beta | \leq N} | x^\alpha \partial_x^\beta f(x) | < \infty 
    \text{ for all } N \in \mathbb N\} ,
\]
and $\mathcal Z_\sigma$ corresponds to the following space,  
\[
\Big\{ f \in \mathcal S(\mathbb R^d) \, \Big| \, 
  \int_{\mathbb R^d} x^\alpha f(x) ~\mathrm{d}x = 0  
  \text{ for all multi-index } \alpha  
\Big\}. 
\]

\noindent 
{\bf Definition. } (Test function spaces and distributions)
\begin{enumerate}
\item[(i)] (Non-homogeneous type) 
We define a test function space $\mathcal X_\sigma$ by 
\[
\mathcal X_\sigma 
:= \{ u \in L^1 \cap L^2_\sigma \, | \, p_M(u) < \infty \text{ for all } M \in \mathbb N\}, 
\]
where 
\[
p_M(u) := \| u \|_{L^1 }+ \sup_{j \in \mathbb N} 2^{Mj} \| \phi_j(\sqrt{A_2}) u \|_{L^1} .
\]

\item[(ii)] (Homogeneous type) 
We define a test function space $\mathcal Z_\sigma$ by 
\[
\mathcal Z_\sigma 
:= \{ u \in \mathcal X_\sigma  \, | \, q_M(u) < \infty 
 \text{ for all } M \in \mathbb N\}, 
\]
where 
\[
q_M(u) := p_M(u) + \sup_{j \leq 0} 2^{M|j|} \| \phi_j(\sqrt{A_2}) u \|_{L^1} . 
\]

\item[(iii)] 
$\mathcal X_\sigma', \mathcal Z'_\sigma$ are the topological duals of 
$\mathcal X_\sigma, \mathcal Z_\sigma$, respectively. 

\end{enumerate}

\vskip3mm

The non-homogeneous and homogeneous spaces will be defined as subspaces of 
$\mathcal X_\sigma ' , \mathcal Z_\sigma '$. 
We give the following definition for 
functions in the Lebesgue spaces belonging 
to $\mathcal X_\sigma ', \mathcal Z_\sigma '$. 

\vskip3mm 

\noindent 
{\bf Definition. } ($L^p$ as subspaces of $\mathcal X'_\sigma, \mathcal Z'_\sigma$)
\quad 
Let $1 \leq p \leq \infty$. 

\begin{enumerate}
\item 
For $f \in L^p$, we define $f$ as an element of $\mathcal X'_\sigma$ by 
\[
{}_{\mathcal X'_\sigma}\langle f, g \rangle _{\mathcal X_\sigma} 
= \int_{\Omega} f(x) \overline{g(x)}~~\mathrm{d}x , 
\quad g \in \mathcal X_\sigma . 
\]

\item 
For $f \in L^p$, we define $f$ as an element of $\mathcal Z'_\sigma$ by 
\[
{}_{\mathcal Z'_\sigma}\langle f, g \rangle _{\mathcal Z_\sigma} 
= \int_{\Omega} f(x) \overline{g(x)}~\mathrm{d}x , 
\quad g \in \mathcal Z_\sigma . 
\]
\end{enumerate}

Analogously to the Dirichlet Laplacian (see~\cite{IMT-2019}), 
we have the following continuous embedding. 
\[
\mathcal Z_\sigma  
\hookrightarrow \mathcal X_\sigma 
\hookrightarrow L^p 
\hookrightarrow \mathcal X_\sigma ' 
\hookrightarrow \mathcal Z_\sigma ' . 
\]

The operator $A_2$ with domain restricted to 
$\mathcal X_\sigma $ or $ \mathcal Z_\sigma$,  defines an operator from 
$\mathcal X_\sigma $ or $ \mathcal Z_\sigma$ to itself.  
We extend $A_2$ to operators on $\mathcal X'_\sigma $ and $ \mathcal Z'_\sigma$, 
which we denote by $A$ in what follows. 

\vskip3mm

\noindent 
{\bf Definition. } (The Stokes operator $A$ on distribution spaces) 
\begin{enumerate}
\item For  $f \in \mathcal X'_\sigma$, we define 
$Af \in \mathcal X'_\sigma$ by 
\[
{}_{\mathcal X'_\sigma}\langle Af, g \rangle _{\mathcal X_\sigma} 
:= {}_{\mathcal X'_\sigma}\langle f, A_2g \rangle _{\mathcal X_\sigma} , 
\quad g \in \mathcal X_\sigma. 
\]

\item For $f \in \mathcal Z'_\sigma$, we define 
$Af \in \mathcal Z'_\sigma$ by 
\[
{}_{\mathcal Z'_\sigma}\langle Af, g \rangle _{\mathcal Z_\sigma} 
:= {}_{\mathcal Z'_\sigma}\langle f, A_2g \rangle _{\mathcal Z_\sigma} , 
\quad g \in \mathcal Z_\sigma. 
\]

\item 
We abbreviate $A_2$ as $A$. 
\end{enumerate}

\vskip3mm 

We also define $\phi(A)$ on $\mathcal X_\sigma '$ and $\mathcal Z_\sigma '$. 

\vskip3mm 

\noindent 
{\bf Definition. } Let $\phi \in C_0^\infty ((0,\infty))$. 
\begin{enumerate}
\item 
For $f \in \mathcal X'_\sigma$, we define 
$\phi (A) f \in \mathcal X'_\sigma$ by 
\[
{}_{\mathcal X'_\sigma}\langle \phi (A) f, g \rangle _{\mathcal X_\sigma} 
:= {}_{\mathcal X'_\sigma}\langle f, \phi (A) g \rangle _{\mathcal X_\sigma} , 
\quad g \in \mathcal X_\sigma . 
\]

\item 
For $f \in \mathcal Z'_\sigma$, we define 
$\phi (A) f \in \mathcal Z'_\sigma$ by 
\[
{}_{\mathcal Z'_\sigma}\langle \phi (A) f, g \rangle _{\mathcal Z_\sigma} 
:= {}_{\mathcal Z'_\sigma}\langle f, \phi (A) g \rangle _{\mathcal Z_\sigma} , 
\quad g \in \mathcal Z_\sigma . 
\]

\end{enumerate}

\subsection{The Projection on Distribution Spaces} \label{subsec:projection}

We define the projection to divergence-free verctor fields 
for functions belonging to $\mathcal X'_\sigma $ 
based on the projection $\mathbb P_2$ on $L^2$.  

\vskip3mm 

\noindent {\bf Definition. } 
\begin{enumerate}
\item 
For $f \in \mathcal X'_\sigma$, we define $\mathbb P f$ 
as an element of $\mathcal X'_\sigma$ by 
\[
{}_{\mathcal X'_\sigma} \langle \mathbb P f , g \rangle _{\mathcal X_\sigma} 
:= 
{}_{\mathcal X'_\sigma} \langle f , \mathbb P_2 g \rangle _{\mathcal X_\sigma} ,
\quad g \in \mathcal X_\sigma. 
\]

\item 
For $f \in \mathcal Z'_\sigma$, we define $\mathbb P f$ 
as an element of $\mathcal Z'_\sigma$ by 
\[
{}_{\mathcal Z'_\sigma} \langle \mathbb P f , g \rangle _{\mathcal Z_\sigma} 
:= 
{}_{\mathcal Z'_\sigma} \langle f , \mathbb P_2 g \rangle _{\mathcal Z_\sigma} ,
\quad g \in \mathcal Z_\sigma.  
\]

\item We abbreviate $\mathbb P_2$ as $\mathbb P$. 
\end{enumerate}

\vskip3mm 

We notice from the definition of $L^1$ as subspaces of 
$\mathcal X'_\sigma, \mathcal Z'_\sigma$ that for $f \in L^1$ 
\[
{}_{\mathcal X'_\sigma} \langle \mathbb P f , g \rangle _{\mathcal X_\sigma} 
:= \int_{\Omega} f(x) \overline{\mathbb P_2 g (x)}~\mathrm{d}x 
= \int_{\Omega} f(x) \overline{ g (x)}~\mathrm{d}x, 
\quad g \in \mathcal X_\sigma. 
\]
\[
{}_{\mathcal Z'_\sigma} \langle \mathbb P f , g \rangle _{\mathcal Z_\sigma} 
:= \int_{\Omega} f(x) \overline{\mathbb P_2 g (x)}\mathrm{d}x
= \int_{\Omega} f(x) \overline{ g (x)}~\mathrm{d}x, 
\quad g \in \mathcal Z_\sigma.  
\]

\subsection{Projection Acting on the Nonlinear Term} \label{proj_nonlinear} 

For the convection term $(u \cdot \nabla ) u$, 
we give the following definition of the nonlinear part as 
an element of $\mathcal X'_\sigma , \mathcal Z'_\sigma$. 

\vskip3mm 

\noindent {\bf Definition. } 
\begin{enumerate}
\item 
For every measurable functions $u, v$ such that $u \otimes v \in L^1$, 
we define $\mathbb P {\rm div \, } (u \otimes v)$ as an element of 
$\mathcal X'_\sigma$ by 
\[
{}_{\mathcal X'_\sigma}\langle 
   \mathbb P {\rm div \, } (u \otimes v), g \rangle _{\mathcal X_\sigma}
:= - \int _{\Omega} (u \otimes v) \cdot \overline{\nabla g }~\mathrm{d}x, 
\quad g \in \mathcal X_\sigma . 
\]

\item
For every measurable functions $u, v$ such that $u \otimes v \in L^1$, 
we define $\mathbb P {\rm div \, } (u \otimes v)$ as an element of 
$\mathcal Z'_\sigma$ by 
\[
{}_{\mathcal Z'_\sigma}\langle 
   \mathbb P {\rm div \, } (u \otimes v), g \rangle _{\mathcal Z_\sigma}
:= - \int _{\Omega} (u \otimes v) \cdot \overline{\nabla g }~\mathrm{d}x, 
\quad g \in \mathcal Z_\sigma .  
\] 
\end{enumerate}

\vskip3mm 

\noindent {\bf Remark. } 
If $u$ satisfies $u \cdot \nu = 0$ on the boundary, we can write 
by integration by parts that formally 
\[
\begin{split}
{}_{L^2}({\rm div \, } (u \otimes u) , g) _{L^2} 
=
& \int_{\Omega} \sum_{j,k = 1}^d \partial _{x_k} (u_j u_k) \cdot \overline{g_j} ~\mathrm{d}x
\\
=& - \int_{\Omega} \sum_{j,k = 1}^d  u_j u_k \overline{\partial_{x_k}g_j} ~\mathrm{d}x 
 + \int_{\partial \Omega} \sum_{j=1}^d u_j \overline{g_j} (u \cdot \nu) ~\mathrm{d}S
\\
=& - \int_{\Omega} \sum_{j,k = 1}^d  u_j u_k \overline{\partial_{x_k}g_j} ~\mathrm{d}x .
\end{split}
\]
We also mention that bilinear estimates cannot be obtained in general 
even for the Dirichlet Laplacian, 
and the validity depends on the regularity index (see the paper~\cite{Iw-2023}). 
The definition above allows us to recognize the nonlinear term 
in spaces of distributions.

\subsection{Besov Spaces}\label{subsec:Besov}

We define Besov spaces associated with the Stokes semigroup 
as subspaces of $\mathcal Z'_\sigma$, 
and show their duals.  
We recall our setting that  $\Omega$ is simply connected, and 
note that $\mathcal X_\sigma =  \mathcal Z_\sigma$, 
since the domain $\Omega$ is bounded, the infimum of the spectrum 
of $A_2$ is positive, and 
$j > j_0$ for some $j_0$ is enough on $j\in \mathbb Z$ 
in the definition of $\mathcal Z_\sigma$. 
Here, we choose the homogeneous type.

\vskip3mm 

\noindent {\bf Definition. }
Let $s \in \mathbb R , 1 \leq p, q \leq \infty$. 
We define the Besov space $\dot B^s_{p,q}$ by 
\[
\dot B^s_{p,q} := 
\Big\{ f \in \mathcal Z'_\sigma \, \Big| \, 
\| f \|_{\dot B^s_{p,q}}
:= \Big\|
  \Big\{ 2^{sj} \| \phi_j(\sqrt{A}) f \|_{L^p} \Big\} _{j \in \mathbb Z}
 \Big\|_{\ell ^q (\mathbb Z)}
 < \infty 
\Big\} . 
\]

\noindent 
{\bf Remark. } When $s = 0$ and $p = q = 2$, we have 
$\dot B^0_{2,2} = L^2_\sigma$. 
If $s \gg 1$, then functions in $\dot B^s_{p,q}$ 
are smooth and 
the divergence-free condtion becomes more apparent. 

\begin{lem}\label{lem:0725-5}
Let $u \in \mathcal Z'_\sigma$. Then we have 
\[
u = \sum_{j \in \mathbb Z} \phi_j(\sqrt{A}) u 
 \text{ in } \mathcal Z'_\sigma  , 
 \quad \text{and} \quad  
\phi_j(\sqrt{A}) u \in L^\infty \text{ for all } j \in \mathbb Z.  
\]
\end{lem}
\begin{pf}
We can follow the proof of Lemma~4.5 in~\cite{IMT-2019} to obtain the 
resolution and $\phi_j(\sqrt{A}) u \in L^\infty$. 
\end{pf}

\begin{lem}
Let $1 < p < \infty$. We have an identity that 
\[
\{ f \in \mathcal Z_\sigma ' \,  \mid \, 
  \| f \|_{L^p} < \infty \} 
  = L^p_\sigma , 
\]
where $L^p_\sigma$ is the closure of $C_{0,\sigma}^\infty$ in $L^p$. 
\end{lem}
\begin{pf}
We know that $C_{0,\sigma}^\infty \subset \{ f \in \mathcal Z_\sigma ' \, \mid \, 
  \| f \|_{L^p} < \infty \}$
 and the completion with respect to the $L^p$ norm implies 
 $L^p_\sigma \subset \{ f \in \mathcal Z_\sigma ' \, | \, 
  \| f \|_{L^p} < \infty \}$.  

The converse follows from the continuity property of 
the Stokes semigroup $\{ e^{-tA} \}_{t\geq 0}$ in $L^p$ 
with $1 < p < \infty$, which is already known (see \cite{Miy-1980}). 
In fact, let $ f \in 
\{ f \in \mathcal Z_\sigma ' \, | \,  \| f \|_{L^p} < \infty \}$. 
It follows from Lemma~\ref{lem:0725-5} that 
\[
\displaystyle f_N := \sum_{j \leq N} \phi_j(\sqrt{A})f
\]
approximates $f$ in $\mathcal Z_\sigma '$, 
belongs to $L^p$ due to the boundedness of the domain $\Omega$,  
and 
\[
\displaystyle \sup _{N \in \mathbb N} \| f_N \|_{L^p} < \infty.
\]  
We write 
\[
\begin{split}
f_N - f 
=& f_N - e^{-tA}f_N
+e^{-tA} f - f 
+ e^{-tA} \sum_{j \geq N+1} \phi_j(\sqrt{A}) f
\\
=& \sum_{j \leq N} \phi_j(\sqrt{A}) (f- e^{-tA} f ) 
+ ( e^{-tA} f - f ) 
+ e^{-tA} \sum_{j \geq N+1} \phi_j(\sqrt{A}) f.
\end{split}
\]
We also have that 
\[
\begin{split}
& 
\Big\| 
\sum_{j \leq N} \phi_j(\sqrt{A}) (f- e^{-tA} f ) 
+ ( e^{-tA} f - f ) 
\Big\| _{L^p} 
\leq C \| e^{-tA} f - f \|_{L^p} \to 0 , \quad t \to 0, 
\\
& 
\Big\| e^{-tA} \sum_{j \geq N+1} \phi_j(\sqrt{A}) f
\Big\|_{L^p}
\leq C \sum _{j \geq N+1} e^{-C^{-1}t2^{2j}} \| f \|_{L^p}
\to 0, \quad N \to \infty, \text{ for each } t. 
\end{split}
\]
Here, we have used Proposition~\ref{prop:0726-2} in the first ineqaulity 
above. 
We have applied the following boundedness (see~Lemma~2.1 of \cite{Iw-2018}) 
of a spectral multiplier multiplied by a smooth function $G$
to the second inequality. 
\[
\| G(\sqrt{A}) \phi_j(\sqrt{A}) \|_{L^p \to L^p} 
\leq C \| G(2^j \sqrt{\cdot}\, )  \phi_0(\sqrt{\cdot}\,) \|_{H^s(\mathbb R)}, 
\]
where $s > d/2 + 1$ and 
$H^s(\mathbb R)$ is the Sobolev space on the real line. 
The above two convergence prove that 
the $L^p$-norm of $f_N -f$ converges to zero. 
Since $f_N \in L^p_\sigma$, we conclude that $f \in L^p_{\sigma}$. 
\end{pf}

\begin{prop}Let $1 < p < \infty$. Then 
$\dot B^0_{p,1} \hookrightarrow L^p_\sigma 
\hookrightarrow \dot B^0_{p,\infty}
$. 
\end{prop}

\begin{pf}
Let $u \in \dot B^0_{p,1}$. By the resolution in Lemma~\ref{lem:0725-5}, 
we have 
\[
\| u \|_{L^p} 
\leq \sum_{j \in \mathbb Z} \| \phi_j(\sqrt{A}) u \|_{L^p} 
= \| u \|_{\dot B^0_{p,1}},
\]
which proves $\dot B^0_{p,1} \hookrightarrow L^p_\sigma $. 
The second embedding is obtained by Proposition~\ref{prop:0726-2}. 
\end{pf}

\begin{prop}\label{prop:0724-1}
Let $s \in \mathbb R, 1 \leq p,q < \infty$. Then 
the dual space of $\dot B^s_{p,q}$ is 
$\dot B^{-s}_{p',q'}$, where $p',q'$ satisfy 
\[
\dfrac{1}{p} + \dfrac{1}{p'} = \dfrac{1}{q} + \dfrac{1}{q'} = 1. 
\]
We also have that 
\[
\| f \|_{\dot B^{-s}_{p',q'}} 
\simeq \sup_{\| g \|_{\dot B^s_{p,q}}=1}  
\big| {}_{\mathcal Z'_\sigma}\langle f, g \rangle _{\mathcal Z_\sigma} 
\big| 
= 
\sup_{\| g \|_{\dot B^s_{p,q}}=1}  
\Big| \sum_{j \in \mathbb Z} \int _{\Omega}
  \phi_j(\sqrt{A}) f \cdot 
  \overline{\Phi_j (\sqrt{A})g} ~\mathrm{d}x 
\Big| ,
\]
where $\Phi_j := \phi_{j-1} + \phi_j + \phi_{j+1}$. 
\end{prop}
\begin{pf}
The duality property is proved in the same way as in 
Proposition~3.1 of \cite{IMT-2019}. 
As for the norm equivalence, 
the first relation follows from the duality property 
\[
\dot B^{-s}_{p',q'} = (\dot B^s_{p,q})^* . 
\]
The second equality  is obtained by the resolution of the identity 
in Lemma~\ref{lem:0725-5} to $f$ 
and the resoluton in $L^2$ for $g$, which also applies 
to $g$ in $\mathcal Z_\sigma$. 
\end{pf}

\subsection{Stokes Semigroup}\label{subsec:StokesSemi}
We introduce the Stokes semigroup and several estimates 
used in the proof of our theorem. 

\vskip3mm 

\noindent {\bf Definition. } 
Let $t \geq 0$. 
We define $e^{-tA}$ as a map from $\mathcal Z_\sigma ' $ to itself by 
\[
{}_{\mathcal Z_\sigma '} \langle e^{-tA} f , g\rangle _{\mathcal Z_\sigma}
= \sum_{j \in \mathbb Z} 
{}_{\mathcal Z_\sigma '} \langle  f , e^{-tA} \phi_j(\sqrt{A})g\rangle _{\mathcal Z_\sigma}, \quad g \in \mathcal Z_\sigma . 
\]

\begin{prop}\label{prop:0726-3}
Let $s \in \mathbb R , 1 \leq p,q \leq \infty$. 
\begin{enumerate}
\item For every $t \geq 0$ and $f \in \dot B^s_{p,q}$, we have 
\[
\| e^{-tA} f \|_{\dot B^s_{p,q}} \leq C \| f \|_{\dot B^s_{p,q}}.
\]
\item 
If $q < \infty$, then we have the continuity that 
for every $f \in \dot B^s_{p,q}$, 
\[
\lim_{ t\to0} \| e^{-tA} f - f \|_{\dot B^s_{p,q}} = 0 .  
\]
If $q = \infty$, then we have the continuity of the weak-$*$ sense that 
for every $f \in \dot B^s_{p,\infty}$,  
\[
\lim_{t \to 0} 
{}_{\dot B^{s}_{p,\infty}}\langle e^{-tA}f -f, g \rangle _{\dot B^{-s}_{p',1}} 
= \lim_{t\to0}\sum_{j \in \mathbb Z} 
 \int_{\Omega} \phi_j(\sqrt{A}) (e^{-tA}f -f)\cdot \overline{\Phi_j(\sqrt{A})g}~dx
 =0,
\]
for all $g \in \dot B^{-s}_{p',1}$, where 
$1/p + 1/p' = 1$ and $ \Phi_j= \phi_{j-1} + \phi_j + \phi_{j+1}$. 
\item 
If $s \geq s_0, p \leq p_0, 
s-s_0 - d(1/p - 1/p_0) > 0$, then 
for every $f \in \dot B^{s_0}_{p_0,\infty}$, 
\[
\| e^{-tA} f \|_{\dot B^s_{p,1}} \leq 
Ct^{-\frac{s-s_0}{2} - \frac{d}{2} (\frac{1}{p} - \frac{1}{p_0}) } \| f \|_{\dot B^{s_0}_{p_0,\infty }}, \quad t > 0 . 
\]
\end{enumerate}
\end{prop}
\begin{pf}
The proof is analogous to that in~\cite{Iw-2018}, which studies the 
heat semigroup for the Dirichlet Laplacian on domains. 
\end{pf}

\begin{prop}\label{prop:0724-2}
Let $1 < p < \infty$. Then there exists a positive constant $C$ such that 
for every $f \in L^p$, 
\[
\|\nabla e^{-tA} f\|_{L^p} 
+ \|e^{-tA} \mathbb P \nabla f\|_{L^p} \leq C t^{-\frac{1}{2}}\| f \|_{L^p}, 
\quad t > 0.
\]
Let $1 \leq p \leq r \leq \infty$. Then there exists a positive constant 
$C$ such that for every $f \in L^p$,
\[
\|e^{-tA} f\|_{L^r} \leq C t^{-\frac{d}{2}(\frac{1}{p}-\frac{1}{r})}\| f \|_{L^p}, 
\quad t > 0. 
\]
\end{prop}

\begin{pf}
It follows from Propositions~\ref{prop:0809-1}, \ref{prop:0726-2} 
(see also Proposition~\ref{prop:0726-3}) that  
\[
\|\nabla e^{-tA} f\|_{L^p} 
\leq C \| A^{\frac{1}{2}} e^{-tA} f \|_{L^p}
\leq C t^{-\frac{1}{2}} \| f \|_{L^p} .
\]
We also have from the duality property 
and the gradient estimate in $L^{p'}$ that 
\[
\begin{split}
\|e^{-tA} \mathbb P \nabla f\|_{L^p}
=& \sup_{g \in \mathcal Z_\sigma, \| g \|_{L^{p'}}=1} 
 \Big| {}_{\mathcal Z'_\sigma} \langle e^{-tA} \mathbb P \nabla f, g
       \rangle _{\mathcal Z_\sigma} 
 \Big| 
= \sup_{g \in \mathcal Z_\sigma, \| g \|_{L^{p'}}=1} 
 \Big| {}_{\mathcal Z'_\sigma} \langle f, \nabla e^{-tA} g
       \rangle _{\mathcal Z_\sigma} 
 \Big| 
\\
\leq 
& C t^{-\frac{1}{2}}\| f \|_{L^{p'}},
\end{split}
\]
where $1/ p + 1/p' = 1$. 
The estimate of the $L^r$ norm by the $L^p$ norm with $ r \leq p$ 
follows from the Gaussian upper bound in Lemma~\ref{lem:0726-1} 
and the Young convolution inequality. 
\end{pf}

\begin{prop}\label{prop:0724-3}
Let $1 < p < \infty$. Then there exists a positive constant $C$ such that 
for every $f \in \dot B^1_{p,1}$,  
\[
\| \nabla f \|_{L^p} \leq C \| f \|_{\dot B^1_{p,1}}. 
\]
\end{prop}
\begin{pf}
We write 
\[
\nabla \phi_j(\sqrt{A})f 
= \nabla e^{-tA} \big( e^{tA}\Phi_j(\sqrt{A}) \big) \phi_j(\sqrt{A})f, 
\quad \text{with } t = 2^{-2j}
\]
and apply Propositions~\ref{prop:0724-2} and \ref{prop:0726-2} to 
$\nabla e^{-tA} $ and $ e^{tA}\Phi_j(\sqrt{A})$, respectively, to have that 
\[
\sum_{j \in \mathbb Z}\| \nabla \phi_j(\sqrt{A})f \|_{L^p}
\leq C \sum _{j \in \mathbb Z} 2^j \| \phi_j(\sqrt{A}) f \|_{L^p},
\]
which proves the inequality of Proposition~\ref{prop:0724-3}. 
\end{pf}

\section{Proof of Theorem~\ref{thm}}

Let $d < p < \infty$. 
We introduce the norm $\| \cdot \|_Y$ defined by 
\[
\| u \|_Y := 
\sup_{t > 0} \| u(t) \|_{\dot B^{-1+\frac{d}{p}}_{p,\infty}}
+ \sup _{t>0} t^{\frac{1}{2}(1-\frac{d}{p})} \| u(t) \|_{L^p} . 
\]
We explain the idea of the proof of the existence of global solutions, 
including some discussion for local solutions. 

The local existence of solutions can be established using the Banach fixed-point theorem (also known as the contraction mapping theorem). We consider the problem in a time interval $[0, T]$ and seek a function $u$ that satisfies the given initial conditions and the integral equation. 
By constructing an appropriate function space and showing that the mapping involved is a contraction, we can guarantee the existence of a local solution.

To obtain a global solution, we use a priori estimates derived from the norm $\| \cdot \|_Y$. This norm controls both the $\dot B^{-1+\frac{d}{p}}_{p,\infty}$ norm and the $L^p$ norm of the solution, ensuring that the solution exists for the entire interval $[0,\infty)$.

Here, we will present the linear and nonlinear estimates related to this 
proof as Step~1 and Step~2, respectively. 
We only present the key estimates and omit the details of applying the fixed point theorem, as the procedure is standard.

\vskip3mm

\noindent 
{\bf Step 1. }
We consider the estimate of the linear term. 
From Proposition~\ref{prop:0726-3} 
and the embedding $\dot B^0_{p,1} \hookrightarrow L^p$, we have  
\[
\| e^{ -tA}u_0 \|_{Y} 
\leq C \| u_0 \|_{\dot B^{-1+\frac{d}{p}}_{p,\infty}} ,
\]
which is enough for the linear part of the proof 
of the existence of global solutions. 

For the existence of local solutions when $q < \infty$, we use the following convergence:
\begin{equation}\label{0829-1}
\lim_{T\to 0}\sup_{t\in(0,T)}t^{\frac{1}{2}(1-\frac{d}{p})}  
 \| e^{-tA}u_{0} \|_{L^p} = 0 . 
\end{equation}
The proof of this is achieved by using the boundedness:
\[
\sup_{t\in(0,T)}t^{\frac{1}{2}(1-\frac{d}{p})}  
 \| e^{-tA}u_{0} \|_{L^p} 
 \leq C \| u_0 \|_{\dot B^{-1+\frac{d}{p}}_{p,\infty}},
\]
and the following approximation of the initial data:
\[
u_{0,N} = \sum_{|j| \leq N} \phi_j(\sqrt{A})u_0.
\]
Since $q < \infty$, it follows that $u_{0,N}$ converges to $u_0$ in 
$\dot B^{-1+\frac{d}{p}}_{p,q}$. 
For the second part of $\| \cdot \|_Y$, we observe that 
\[
\lim_{T\to 0}\sup_{t\in(0,T)}t^{\frac{1}{2}(1-\frac{d}{p})}  
 \| e^{-tA}u_{0,N} \|_{L^p} = 0 , \quad \text{for each } N.
\]
Thus, the proof of \eqref{0829-1} is complete.

When $q = \infty$, instead of the convergence \eqref{0829-1}, 
we consider replacing by the following smallness condition:  
\[
\limsup_{T\to 0}\sup_{t\in(0,T)}t^{\frac{1}{2}(1-\frac{d}{p})}  
 \| e^{-tA}u_{0} \|_{L^p} \ll 1 .
\]

\vskip3mm 

\noindent 
{\bf Step 2. }
We establish the following estimate for the nonlinear term: 
\[
\Big\| 
\int_0^t e^{-(t-\tau)A} \mathbb P {\rm div \,} (u \otimes u) ~d\tau
\Big\|_{Y} 
\leq C \| u \|_{Y}^2 .
\]
Instead of the original norm of the Besov space, 
we  employ the representation via duality. 
To apply Proposition~\ref{prop:0724-1}, 
for every $g \in \dot B^{1-\frac{d}{p}}_{p',1}$, we consider 
the coupling of the nonlinear term (see subsection~\ref{proj_nonlinear}) and $g$ and write 
\[
\begin{split}
\Big| {}_{\mathcal Z'_\sigma} 
\langle  
  \int_0^t e^{-(t-\tau)A} \mathbb P {\rm div \,} (u \otimes u) ~\mathrm{d}\tau, 
g \rangle _{\mathcal Z_\sigma} 
\Big| 
= 
& 
\Big| 
\int_0^t 
  \int_{\Omega} u \otimes u \cdot 
  \overline{ {\rm div \, } \mathbb P e^{-(t-\tau) A} g }
  ~\mathrm{d}x\,\mathrm{d}\tau 
\Big| 
\\
\leq & 
\int _0^t \| u \otimes u \|_{L^{\frac{p}{2}}} 
 \|{\rm div \, } \mathbb P e^{-(t-\tau) A} g\|_{L^{(\frac{p}{2})'}}
 \mathrm{d}\tau 
\\
\leq & \| u \|_Y^2 
\int _0^t \tau ^{-\frac{1}{2}(1-\frac{d}{p}) \cdot 2} 
 \|{\rm div \, } \mathbb P e^{-(t-\tau) A} g\|_{L^{(\frac{p}{2})'}}
 \mathrm{d}\tau .
\end{split}
\]
We know that $\mathbb P e^{-(t-\tau) A} g = e^{-(t-\tau) A} g$, 
the derivatives are bounded by the Besov norm with additional regularity 
from Proposition~\ref{prop:0724-3},  
and 
the semigroup estimate from $L^{p'}$ to $L^{(\frac{p}{2})'}$ from 
Proposition~\ref{prop:0724-2}. 
These imply that for every $t > 0$, 
\[
\begin{split}
\int _0^t \tau ^{-\frac{1}{2}(1-\frac{d}{p}) \cdot 2} 
 \|{\rm div \, } \mathbb P e^{-(t-\tau) A} g\|_{L^{(\frac{p}{2})'}}
 \mathrm{d}\tau 
\leq & 
C \int _0^t \tau ^{-\frac{1}{2}(1-\frac{d}{p}) \cdot 2} 
 \| e^{-(t-\tau) A} g\|_{\dot B^1_{(\frac{p}{2})',1}}
 \mathrm{d}\tau 
\\
\leq
& C\int_0^t \tau ^{-1+\frac{d}{p}} (t-\tau) ^{-\frac{d}{2p} -\frac{d}{2}(\frac{1}{p'}-\frac{1}{(\frac{p}{2})'})} 
 \| g \| _{\dot B^{1-\frac{d}{p}}_{p',1}}\mathrm{d}\tau 
\\
\leq
& C
 \Big( \int_0^1 \tau ^{-1+\frac{d}{p}} (1-\tau) ^{-\frac{d}{p}} \mathrm{d}\tau 
 \Big) \| g \| _{\dot B^{1-\frac{d}{p}}_{p',1}} . 
\end{split}
\]
It follows from the two estimates above and Proposition~\ref{prop:0724-1} that 
\[
\Big\| \int_0^t e^{-(t-\tau)A} \mathbb P {\rm div \,} (u \otimes u) ~d\tau
\Big\|_{\dot B^{-1+\frac{d}{p}}_{p,\infty}} 
\leq C \| u \|_Y ^2 . 
\]

We now consider the second part of $\| \cdot \|_Y$. 
We apply the semigroup estimate 
from $L^{\frac{p}{2}}$ to $L^p$ along with the first derivative 
provided in Proposition~\ref{prop:0724-2}, 
which implies that 
\[
\begin{split}
t^{\frac{1}{2}(1-\frac{d}{p})}
\Big\| \int_0^t e^{-(t-\tau)A} \mathbb P {\rm div \,} (u \otimes u) ~\mathrm{d}\tau
\Big\|_{L^p} 
\leq 
&C t^{\frac{1}{2}(1-\frac{d}{p})}
\int_0^t (t-\tau)^{-\frac{d}{2}(\frac{2}{p}-\frac{1}{p}) - \frac{1}{2}} 
    \| u \otimes u \|_{L^{\frac{p}{2}}} ~\mathrm{d}\tau 
\\
\leq
& C t^{\frac{1}{2}(1-\frac{d}{p})} 
  \int_0^t (t-\tau)^{-\frac{d}{2p}-\frac{1}{2}}
    \tau ^{-\frac{1}{2}(1-\frac{d}{p})\cdot 2} \mathrm{d}\tau 
  \| u \|_Y^2 
\\
  =
& C 
  \Big( \int_0^1 (1-\tau)^{-\frac{d}{2p}-\frac{1}{2}}
    \tau ^{-1+\frac{d}{p}} \mathrm{d}\tau 
  \Big) 
  \| u \|_Y^2 .
\end{split}
\]

\vskip10mm


\vskip3mm 
%
\noindent 
{\bf Conflict of Interest. }
The authors declare that he has no conflict of interest. 

\vskip2mm 

\noindent 
{\bf Data availability. } 
Our manuscript has no associated data. 
%
%

\end{document}